\documentclass[a4paper,reqno]{amsart}

\usepackage{amsthm,amsmath,amssymb}
\usepackage[utf8]{inputenc}
\usepackage{mathtools}
\usepackage{bm}
\usepackage{xcolor} \definecolor{bluegreen2}{RGB}{0, 85, 127}
\usepackage[linktocpage]{hyperref}
\hypersetup{breaklinks, colorlinks=true, linkcolor=bluegreen2, citecolor=bluegreen2, filecolor=magenta, urlcolor=bluegreen2}
\usepackage{tikz}
\usetikzlibrary{arrows.meta, cd}
\usepackage{enumitem} \setlist{itemsep=1.5mm}
\usepackage{newunicodechar}

\theoremstyle{plain}
\newtheorem{theorem}{Theorem}[section]
\newtheorem{prop}[theorem]{Proposition}
\newtheorem{cor}[theorem]{Corollary}
\newtheorem{lemma}[theorem]{Lemma}
\theoremstyle{definition}

\newtheorem{example}[theorem]{Example}

\theoremstyle{remark}
\newtheorem{remark}[theorem]{Remark}

\newcommand{\RR}{\mathbb{R}}
\newcommand{\CC}{\mathbb{C}}
\newcommand{\NN}{\mathbb{N}}
\newcommand{\ZZ}{\mathbb{Z}}

\newcommand{\PP}{\mathbb{P}}
\newcommand{\LL}{\mathbb{L}}
\newcommand{\one}{\mathbf{1}}
\newcommand{\zero}{\mathbf{0}}
\newcommand{\ba}{{\mathbf{a}}}
\newcommand{\KVarC}[1]{K^{#1}_0(\Var_\CC)}

\newcommand{\Zmot}[1]{{Z^{#1}_{\mot}}}
\newcommand{\Znv}[1]{{Z^{#1}_{\nv}}}
\newcommand{\Ztop}[1]{{Z^{#1}_{\topo}}}
\newcommand{\termino}[3]{\str^{#1}_{#2}(#3)}
\newcommand{\abs}[1]{\left\lvert#1\right\rvert}
\newcommand{\ex}{Q}

\DeclareMathOperator{\str}{W}
\DeclareMathOperator{\ord}{ord}
\DeclareMathOperator{\ac}{ac}
\DeclareMathOperator{\Var}{\textbf{Var}}
\DeclareMathOperator{\id}{id}
\DeclareMathOperator{\mot}{mot}
\DeclareMathOperator{\topo}{top}
\DeclareMathOperator{\nv}{naive}
\DeclareMathOperator{\dif}{d\!}
\DeclareMathOperator{\lcm}{lcm}
\DeclareMathOperator{\arco}{\mathcal{A}}
\DeclareMathOperator{\pol}{Pol}

\title[On a suspension formula for Denef-Loeser zeta functions]{On a suspension formula for Denef-Loeser zeta functions}

\author[E.~Artal]{Enrique Artal Bartolo}
\address{Departamento de Matemáticas, IUMA, Universidad de Zaragoza, C. Pedro Cerbuna 12,
50009, Zaragoza, Spain.
}
\email{\href{mailto:artal@unizar.es}{artal@unizar.es}}

\author[P.~Gonz\'alez P\'erez]{Pedro D. Gonz\'alez P\'erez}
\address{Instituto de Matemática Interdisciplinar,
Departamento de Álgebra, Geometría y Topología,
Facultad de Ciencias Matemáticas,
Universidad Complutense de Madrid,
Plaza de las Ciencias 3, 28040, Madrid, Spain
}
\email{\href{mailto:pdperezg@ucm.es}{pdperezg@ucm.es}}

\author[M.~Gonz\'alez Villa]{Manuel Gonz\'alez Villa}
\address{
Centro de investigaci\'on en
Matem\'aticas\\
Apartado Postal
402,
C.P. 36000, Guanajuato, GTO, M\'exico.}
\email{\href{mailto:manuel.gonzalez@cimat.mx}{manuel.gonzalez@cimat.mx}}

\address{
\indent Departamento de Matemáticas, IUMA, Universidad de Zaragoza,\\ C. Pedro Cerbuna 12,
50009, Zaragoza, Spain.
}
\email{\href{mailto:m.gonzalez@unizar.es}{m.gonzalez@unizar.es}}

\author[E.~Le\'on Cardenal]{Edwin Le\'on Cardenal}
\address{Departamento de Matem\'{a}ticas y Computación \\
Universidad de La Rioja \\
C. Madre de Dios 53\\
26006 Logroño, Spain.}
\email{\href{mailto:edwin.leon@unirioja.es}{edwin.leon@unirioja.es}}

\thanks{
EAB, MGV, and ELC were partially supported by MCIN/AEI/10.13039/501100011033 (grant code: PID2020-114750GB-C31)
and by Departamento de Ciencia, Universidad y Sociedad del Conocimiento del Gobierno de Arag{\'o}n
(grant code: E22\_20R: ``{\'A}lgebra y Geometr{\'i}a''). MGV and ELC were also partially supported by the European Union NextGenerationEU/PRTR and UNIZAR via María Zambrano's Program and by CONAHCYT project CF-2023-G33. PGP and MGV were partially supported by MCIN/AEI/10.13039/501100011033 (grant code: PID2020-114750GB-C33)
}


\begin{document}

\begin{abstract}
Formulas for the topological zeta functions of suspensions by 2 points are due to Artal \emph{et al.}
We generalize these formulas to the motivic level and for arbitrary suspensions, by using a stratification principle
and classical techniques of generating functions in toric geometry. The same strategy is used to obtain formulas
for the motivic zeta functions of some families of non-isolated singularities related to superisolated, Lê-Yomdin, and weighted Lê-Yomdin singularities.
\end{abstract}

\maketitle

\thispagestyle{empty}

\tableofcontents

\section*{Introduction}

The original goal of this note is to compute the naive motivic local zeta function $\Znv{}(F, T)_\zero$  and its topological specialization, where
$F(x_1, \dots, x_d, z) = z^\ex - f(x_1, \ldots, x_d)\in\CC\{x_1,\dots,x_d,z\}$ is the suspension of $f\in\CC\{x_1,\dots,x_d\}$ by $\ex$ points. These functions
define germs of hypersurface singularities in $(\CC^{d+1},\zero)$ and $(\CC^d,\zero)$. The naive zeta function for $F$ is
given by
\[
\Znv{}(F, T)_\zero:=\sum_{n \in \ZZ_{>0}} \mu(\mathcal{X}_n) T^n,
\]
where $\mathcal{X}_n$ stands for the set of arcs $\varphi \in \mathcal{L}(\CC^{d+1})_\zero$ such that $\ord F \circ \varphi =n$.

A formula for the local topological zeta function $\Ztop{}(F,s)_\zero$ of the suspension  by $N=2$ points is given in  \cite[Theorem 1.1]{ACNLM-JLMS}.  As an application of the formula, the authors show that the topological zeta function is not a topological invariant of the singularity. Moreover,
a correspondent pointed out to the authors of~\cite{ACNLM-JLMS} a formula for the suspension by $\ex$ points, based upon the motivic Thom-Sebastiani theorem, see \cite[\emph{Note added in proof}, p.~53]{ACNLM-JLMS}. This generalization has apparently been unquestioned for more than two decades now, but it is inaccurate as we will see in Section~\ref{comp}.

One of the main results of this paper is the generalization of the formula in \cite[Theorem 1.1]{ACNLM-JLMS}. It is a double generalization, since
we consider arbitrary~$\ex$ instead of only $\ex=2$ and it is done for $\Znv{}(F, T)_{\zero}$ and not only for the topological zeta function. After specialization, we get the following explicit expression for the local topological zeta function
\begin{align}\label{fsuspACNLM}
	\Ztop{}(F,s)_\zero&= \frac{1}{\ex t}+
	\frac{\ex-1}{\ex}
	\cdot \frac{s}{s+1}
	\cdot \frac{t + 1}{t}
	\cdot \Ztop{}\left(f,t\right)_\zero- \frac{s}{s+1} \sum_{1\neq e \mid \ex} \frac{J_2(e)}{\ex} \Ztop{(e)}\left(f,t\right)_\zero,
\end{align}
where $t=s+\frac{1}{\ex}$,  $J_2 : \NN \rightarrow \NN$ is a Jordan's totient function and $\Ztop{(e)}$ are the twisted topological zeta functions. If we want to iterate the suspensions, we also need a formula for the twisted topological zeta functions. Expression \eqref{fsuspACNLM} and similar expressions for the twisted zeta functions can be found in Theorem~\ref{thmfsuspACNLM}. Even more, this result allows a more general setting of a pair $(F,\omega_{d+1})$ with $\omega_{d+1}$ a monomial holomorphic $(d+1)$-form.

While these formulas can be seen as a particular case of a Thom-Sebastiani formula, we give another generalization beyond that setting. Namely, we will compute the zeta functions
for $G(x_1, \dots, x_d, z):=z^p F(x_1, \dots, x_d, z)$ and a monomial $(d+1)$-differential form. This situation already appeared in~\cite{ACNLM-ASENS}, for the computation of the topological zeta function
of superisolated singularities. The key point there was the computation of the topological zeta function of the pair $(G,\omega_{d+1})$ where the values $\ex=1$ and $p>0$ are chosen for $G$, and the differential form is $\omega_{d+1}=z^d\cdot \dif x_1\wedge\cdots\wedge\dif x_d\wedge\dif z$.

Let us now describe some possible applications of our results. As pointed to the first author by S.~Gusein-Zade, the computation of the local topological zeta functions for Lê-Yomdin singularities~\cite{Yomdin74, Le80} reduces by the Stratification principle of  \cite{ACNLM-ASENS} to the computation of $\Ztop{}(G,\omega_{d+1},s)_\zero$ for a function $G$ with arbitrary values $p$ and $\ex$ and $\omega_{d+1}$ as above. This is precisely the content of our Theorem~\ref{thmpsuspACNLM}. The computation of A'Campo's monodromy zeta function for Lê-Yomdin singularities in~\cite{GLM:97, jmm:14}, the results in~\cite{ACNLM-ASENS} and Theorem~\ref{thmpsuspACNLM} are the main tools to attack the monodromy conjecture for those singularities. It will also give an alternative proof of the monodromy conjecture for superisolated singularities, a particular case of Lê-Yomdin singularities. Moreover, all these results may be generalized for quotient singularities using the techniques in this work and the results of~\cite{lmvv:20}. This would help to compute the
zeta functions for weighted Lê-Yomdin singularities, see~\cite{ABLM-milnor-number,ACMNemethi}.

The \texttt{Sagemath} package \cite{viu:22} has been particularly useful during the research on this project.

The paper is organized as follows. Section \ref{sec:settings} is devoted to introduce general settings and notations, including some arithmetic function and the basics of arc spaces and motivic/topological zeta functions. The computations of local motivic and topological zeta functions associated to a particular stratum are done in Sections \ref{sec:motivic_zeta_binomials} and \ref{sec:top_zeta_binomials}. Section \ref{sec:motivic_zeta_susp} explains how to use the Stratification principle to reduce the main results to the local computations of Sections \ref{sec:motivic_zeta_binomials} and \ref{sec:top_zeta_binomials}.  The main results are stated and proved in Section \ref{sec:MainResults}. Finally, we compare Theorem~\ref{thmfsuspACNLM} with the formula proposed in the aforementioned \emph{Note added in proof} in Section \ref{comp}.

\numberwithin{equation}{section}

\section{General settings and notations}
\label{sec:settings}

\subsection{Arithmetic functions}
\mbox{}

Denote the set of positive integers by $\NN$. Consider the commutative ring of arithmetic functions  $(A, +, \ast)$ of functions $f: \NN \rightarrow \NN$ such that $f(1) \ne 0$ with the operations of pointwise addition $+$ and the Dirichlet convolution $\ast$ that is given in terms of the usual product $\cdot$ of the integers  by
\[
(f \ast g)(n)= \sum_{d \mid n}  f(d) \cdot  g(n/d).
\]
There are some simple examples of arithmetic functions:

\begin{itemize}
\item The constant function
$\one: \NN \rightarrow \NN$ given  by  $\one(n)=1$ for any positive integer $n$.

\item The
\emph{unit} function $\epsilon  : \NN \rightarrow \NN$  given by $\epsilon (1) = 1$ and $\epsilon(n) =0 $ if $n \ne 1$.
Note that $\epsilon \ast f = f \ast \epsilon =f$ for any $f \in A$.

\item For any positive integer $k$, the functions $\sigma_k   : \NN \rightarrow \NN$, given by $\sigma_k(n) = n^k$ for any positive integer $n$.

\item The \emph{Möbius function}  $\mu  : \NN \rightarrow \NN$ is the function that, given a positive integer $n$,
$\mu(n)$ is the sum of the primitive $n$th roots of unity. It is well-defined since
the cyclotomic polynomials are monic polynomials over the integers.

\item The \emph{Euler's totient function} $\phi  : \NN \rightarrow \NN$ is the function that, given a positive integer $n$,
counts the number of
coprime residues modulo $n$. It satisfies the famous Gauss' identity
\begin{equation}\label{eq:Gauss}
\sum_{d \mid n} \phi(d)= n
\end{equation}

\item For any positive integer $k$, the  \emph{Jordan's totient function}  $J_k : \NN \rightarrow \NN$ is the function that,  given a positive integer $n$, counts the $k$-tuples of positive integers all less than or equal to $n$
that form a coprime $(k+1)$-tuple together with $n$.
The Jordan's totient functions are a generalization of  Euler's totient function since $J_1 = \varphi$.

\end{itemize}

\begin{remark} The values of  $J_k$  are given by the expression
\[
J_k(n) = n^k \prod_{\substack{p \mid n\\p\text{ prime}}}
\left(1 - \frac{1}{p^k}\right),
\]
Some values of $J_2$ are listed in OEIS~\cite{oeis007434}.
 \end{remark}

 The following Lemma will be useful to give explicit expressions for the Denef-Loeser zeta functions of a suspension.

\begin{lemma}
 The Jordan's totient functions $J_k$ and the functions $\sigma_k$ are related under Dirichlet convolution by the following expressions
\[
\sigma_k = \one \ast J_k, \text{ and }
J_k = \mu \ast \sigma_k,
\]
or, equivalently, by the expressions
\begin{equation}\label{jk}
 n^k = \sum_{d \mid n} J_k(d),  \text{ and }
 J_k(n) = \sum_{d \mid n} \mu(d) \left(\frac{n}{d}\right)^k
 .
\end{equation}
\end{lemma}

\subsection{Arc spaces}
\mbox{}

Let $\varphi$ be an element  in $\CC [[t]] \setminus \{0\}$, we denote by ${\rm ord}(\varphi)$  the \textit{order}
(with respect to~$t$) of the series  $\varphi$, and by $\mathrm{ac}(\varphi)$ the coefficient of the leading term of $\varphi$, i.e., if $\varphi= a_n t^n + a_{n+1}t^{n+1} + \cdots$ and $a_n \ne 0$, then ${\rm ord}(\varphi)=n$ and $\mathrm{ac}(\varphi)=a_n$.

Let us fix a system of coordinates $(x_1, \dots, x_d, z)$ of $\CC^d \times \CC = \CC^{d+1}$ centered at $\zero \in \CC^{d+1}$.
Denote by $\varphi \in \mathcal{L}_n(\CC^{d+1})_\zero$ the space of $n$-jets $\varphi$ of $\CC^{d+1}$ centered at $\zero$, i.e. such that $\varphi(0)=\zero$.
Denote the coordinates of $\varphi \in \mathcal{L}_n(\CC^{d+1})_\zero$ as $\varphi_{x_i}$ or $\varphi_z$, i.e., $\varphi=(\varphi_{{\bf x}}, \varphi_z)$.

\subsection{Grothendieck rings and their specializations}
\mbox{}

Let us denote by $\Var_{\CC}$ the category of complex algebraic varieties, i.e.,
reduced and separated schemes of finite type over the complex field $\CC$. Let us denote by  $\KVarC{}$ the \textit{Grothendieck group} of algebraic varieties over $\CC$, and  by the symbol $[X]$ the class of the
complex algebraic variety  $X$ in $\KVarC{}$.
Denote by
${\LL}$ the class of ${\mathbb{A}}^1_\CC$ in $\KVarC{}$ and by $\mathcal{M}_\CC$
the ring $\KVarC{}[\LL^{-1}]$.
We also consider the category of complex algebraic varieties endowed with a good action of the profinite group $\hat{\mu}$ of the roots of the unity. The corresponding Grothendieck rings are denoted by adding the superscript $\hat{\mu}$ to the usual symbols.

Denote by $\chi_{\topo}: \KVarC{}[\LL^{-1}] \rightarrow \ZZ$ the usual Euler characteristic and by $\chi_{\topo}(, \alpha) : \KVarC{\hat{\mu}}[\LL^{-1}] \rightarrow \ZZ$ the equivariant Euler characteristic with respect to a character $\alpha: \hat{\mu} \rightarrow\CC^* $ of finite order.

\subsection{Denef-Loeser zeta functions}
\label{subsec:dlzf}
\mbox{}

Motivic zeta functions can be generalized
to take into account volume forms $\omega$. We fix a volume form $\omega\in\Omega^\ell(\CC^\ell)$. Then the local naive motivic zeta function
and the local motivic zeta function of a germ $g:(\CC^\ell,\zero)\to(\CC,0)$ at~$\zero$ with respect to $\omega$ are defined as
\begin{equation}\label{eq:zetamot}
\begin{aligned}
\Znv{}(g, \omega, T)_\zero&:=
\sum_{n\in \NN}\sum_{m \in \ZZ_{\geq 0}} [\mathcal{X}_{n,m}]{\LL}^{-n\ell -m}T^n \in {\mathcal{M}}_\CC[[T]],\\
\Zmot{}(g, \omega, T)_\zero&:=\sum_{n\in \NN}\sum_{m \in \ZZ_{\geq 0}} [\mathcal{X}^1_{n, m }]{\LL}^{-n\ell-m}T^n \in {\mathcal{M}}^{\hat{\mu}}_\CC[[T]].
\end{aligned}
\end{equation}
Here $\mathcal{X}_{n,m}$ stands for the set of $n$-jets $\varphi \in \mathcal{L}_n(\CC^{\ell})_\zero$
such that $\ord g \circ \varphi =n$
and $\ord \varphi^*\omega =m$; $\mathcal{X}^1_{n,m}$ stands for the set of $n$-jets in $\mathcal{X}_{n,m}$
such that $\ac(g \circ \varphi)=1$. We also consider  the local topological zeta function $\Ztop{}(g,\omega,s)_\zero$
and the local twisted topological zeta function  $\Ztop{(e)}(g,\omega,s)_\zero$ of $g$ at $\zero$ given by
\begin{equation}\label{eq:zetatop}
\chi_{\topo}(\Znv{}(g,\omega, \LL^{-s})_\zero), \text{ and }
\chi_{\topo}((\LL -1)\Zmot{}(g,\omega, \LL^{-s})_\zero, \alpha),
\end{equation}
where $e$ is an integer and  $\alpha: \hat{\mu} \rightarrow \CC^*$ is a character of order $e$.

Let us consider an embedded  resolution of $\pi : Y \rightarrow \CC^\ell$ of $g\omega$. Let $\{E_j\mid j\in J\}$ be the irreducible components of $\pi^* g\omega$, including
exceptional and strict transform components.
As divisors, we write
\begin{equation*}
E:=(g\circ\pi)^{*}(0)=\sum_{j\in J} N_j E_j,
\quad
(\pi^*\omega)=\sum_{j\in J} (\nu_j-1) E_j
.
\end{equation*}
The intersection of the divisor $E$ with $\pi^{-1}(\zero)$ is stratified by
\begin{equation*}
\{E^\circ_I\mid \emptyset\neq I \subset J\},\qquad
E^\circ_I := \bigcap_{i \in I} E_i \setminus \bigcup_{j \not \in I} E_j.
\end{equation*}
For $I\subset J$ and a point $o_I \in E_I^\circ$ there is a suitable coordinate system
$\mathcal{U}$ in the Zariski topology with two subsets of coordinates. On one side coordinates $x_i$, $i\in I$, such that $E_i$ is defined by $x_i=0$; on the other side a $d-\abs{I}$ system of coordinates $y_1,\dots,y_{d-\abs{I}}$ which parametrizes $E_I$.
Moreover,
the pullbacks of $f$ and $\omega$ under $\pi$ will be denoted by $g_{I}$ and $\omega_{I}$, they have local equations \[
f\circ\pi=u\prod_{i\in I}x_i^{N_i},\quad
\pi^*\omega=v\prod_{i\in I}x_i^{\nu_i} \cdot
\bigwedge_{i\in I}\frac{\dif x_i}{x_i}\wedge\dif y_1\wedge\dots\wedge\dif y_{d-\abs{I}}
\]
where $u,v$ are non-vanishing functions in $\mathcal{U}$ depending on the variables
$x_i$, for $i\in I$, and $y_1,\dots,y_{d-\abs{I}}$.

For $I\subset J$, we denote
$g_I:=g\circ\pi$ in a neighbourhood of $E_I^\circ$.
For any point in $E_I^\circ$ we can choose
coordinates in an open set of the origin in  $\CC^{I}\times\CC^{d-\abs{I}}$
such that $g_I$ is the product of the monomial $\prod_{i\in I} x_i^{N_i}$
and a unit which may depend on all the variables.
We consider also
the unramified Galois cover $\tilde{E}_I^\circ$ of $E_I^\circ$
introduced
in~\cite[\S3.3]{Denef-LoeserBarca}.

The following formulas hold:
\begin{equation}\label{eq:zetas_res}
\begin{aligned}
\Znv{}(g, \omega, T)_\zero&=
\sum_{\emptyset \ne I \subset J}
[E_I^\circ] \prod_{i\in I}\frac{\LL-1}{\LL^{\nu_i}-T^{N_i}}T^{N_i} \in {\mathcal{M}}_\CC[[T]],\\
\Zmot{}(g,\omega, T)_\zero&=
\sum_{\emptyset \ne I \subset J}
[\tilde{E}_I^\circ] \prod_{i\in I}\frac{\LL-1}{\LL^{\nu_i}-T^{N_i}}T^{N_i} \in {\mathcal{M}}^{\hat{\mu}}_\CC[[T]],\\
\Ztop{}(g,\omega,s)_\zero&=
\sum_{\emptyset \ne I \subset J} \frac{ \chi( E_I^\circ)}{\prod_{i \in I} (N_is + \nu_i)}, \\
\Ztop{(e)}(g,\omega, s)_\zero&=\sum_{\substack{\emptyset \ne I \subset J\\ e \mid N_i}} \frac{ \chi( E_I^\circ)}{\prod_{i \in I} (N_is + \nu_i)}.
\end{aligned}
\end{equation}
According to \cite{DL-JAMS} the local topological zeta function of a regular function $f$ satisfies  $\Ztop{}(f,0)_\mathbf{0} =1$. This identity can be rephrased,  using the expression of the local zeta function in terms of an embedded resolution \eqref{eq:zetas_res}, as
\begin{equation}\label{DLz(0)=1}
\sum_{\emptyset \ne I \subset J}  \frac{\chi(E_I^\circ) }{\displaystyle\prod_{i \in I}  \nu_i} = 1.
\end{equation}
For the sake of completeness we state a generalization of the Stratification principle,
that will we be used in~\S\ref{sec:motivic_zeta_susp}.

\begin{prop}[{\cite[Stratum principle 1.2]{ACNLM-ASENS}}]\label{prop:sp}
Let $X$ be a smooth algebraic variety, and
$f : X \rightarrow \CC$ be a regular function.
Let $X = \bigcup_{S \in\mathcal{S}} S$ be a finite prestratification of $X$ such that for each $x \in X$, the local naive zeta function at $x$ depends  only on the stratum $S$ containing $x$. Let us denote by $\Znv{}(f, \omega, T; S)_{x}$ the common zeta function associated with the stratum $S$. Then,
\[
\Znv{}(f,\omega,T)_x = \sum_{S \in \mathcal{S}} \Znv{}(f, \omega, T ; S)_{x}.
\]
\end{prop}
The above principle extends to the local motivic zeta function, and implies similar statements for the local topological and local twisted topological zeta functions.

\subsection{General setting}
\label{subsec:general}
\mbox{}

Consider a regular function $f:(\CC^d,\zero) \rightarrow (\CC,0)$ and a fixed $\ex \in \ZZ_{>0}$. The suspension of $f$ by $\ex$ points is defined by the regular function
$F: (\CC^{d+1}, \zero) \rightarrow (\CC, 0)$, given by the expression
\[
F(x_1, \dots, x_d, z) = z^\ex - f(x_1, \ldots, x_d).
\]
With applications in mind we will also consider  $G(x_1, \dots, x_d, z):=z^p F(x_1, \dots, x_d, z)$, for
$p\geq0$.

\subsection{Notations}
\mbox{}

Let $I$ be a finite totally ordered set and let $H$ be an algebraic object  (e.g. $\NN,\ZZ,\RR,\dots$).
The canonical
basis of $H^I$ is denoted by $\mathbf{e}_k$, with
$k\in I$. An element $\mathbf{v}\in H^I$ is written as $\sum_{k \in I} v_k \mathbf{e}_k$, where
$v_k$ is called  the coordinate
of $\mathbf{v}$ in $\mathbf{e}_k$. We denote by $\abs{\mathbf{v}}$
the sum $\sum_{k \in I} v_k$ of the coordinates in the canonical basis.

If $\mathbf{v}\in H^I$ and $L\subset I$,
then we denote by  $\mathbf{v}_L\in H^L$ the \emph{vector} obtained
by deleting the entries corresponding to $I\setminus L$.
Sometimes we will need a superset of $I$ by adding an extra element~$z$ as the last element;
the superset $I\cup\{z\}$ is denoted by $I_z$. The extra element in the canonical basis
of $H^{I_z}$ is denoted by $\mathbf{e}_z$ and the corresponding coordinate of a vector
$\mathbf{v}\in H^{I_z}$ is denoted by $v_z$.

\section{Naive and motivic Zeta functions for certain binomials}
\label{sec:motivic_zeta_binomials}

In \S\ref{sec:motivic_zeta_susp} we are going to use the Stratification principle
in the motivic setting in order to compute the naive and motivic zeta
functions of $F$ and $G$. In this section we give an explicit formula for the contributions
of the different strata coming from the resolution of $f$.

Let us fix a finite set $I$ and vectors
$\mathbf{N}\in \NN^{I_z}$, $\boldsymbol{\nu} \in \ZZ_{\geq 0}^{I_z}$; we denote $p:=N_z$. Let $\ex$ be a positive integer.
Consider $g_0(\mathbf{x},z):=z^\ex - \mathbf{x}^{\mathbf{N}_I}$
and set
\begin{align}
\label{eq:g}
g(\mathbf{\boldsymbol{x}},z):=& z^p g_0(\mathbf{x},z) =
z^p \left(z^\ex - \prod_{k\in I} x_k^{N_k}\right),\\
\label{eq:omega}
\omega :=& \mathbf{\boldsymbol{x}}^{\boldsymbol{\nu}} z^{\nu_z} \frac{\dif \mathbf{\boldsymbol{x}}}{\mathbf{x}} \frac{\dif z}{z}= \prod_{k\in I}
x_k^{\nu_k} z^{\nu_{z}}  \left(\prod_{k\in I}\frac{\dif x_k}{x_k}\right)\frac{\dif z}{z}.
\end{align}
We restrict ourselves to the space $\arco$  of arcs based at $\mathbf{0}$ not belonging to the arc space of the union of the coordinate hyperplanes and the hypersurface $V(g)$.
To such an arc $\varphi$, we associate vectors $\ba\in(\CC^*)^{I_z}$ and $\ord\varphi:=\mathbf{b}\in\NN^{I_z}$ determined by the expressions
\[
\ord\varphi_k = b_k \quad \text{  and }  \quad \ac(\varphi_k) = a_k, \quad  \text{i.e.}, \quad  \varphi_k = a_kt^{b_k} + \cdots  \, \text{ for  } k\in I_z.
\]
We consider a partition of $\arco$ in four subspaces:
\begin{align*}
\arco_{\sigma^+}\!\!&:=\{\varphi\in\arco\mid\ord(z^{\ex+p} \circ \varphi) > \ord (\mathbf{x}^{\mathbf{N}_I}z^p\circ \varphi)=\ord(g \circ \varphi)\}=
\{\varphi\in\arco\mid(\ex + p)b_{z} > \langle\mathbf{b},\mathbf{N}\rangle\}\\
\arco_{\sigma^-}\!\!&:=\{\varphi\in\arco\mid\ord(g \circ \varphi)=
\ord(z^{\ex+p} \circ \varphi) < \ord (\mathbf{x}^{\mathbf{N}_I}z^p\circ \varphi)\}=
\{\varphi\in\arco\mid(\ex + p)b_{z} < \langle\mathbf{b},\mathbf{N}\rangle\}\\
\arco_{\rho}\ &:=\{\varphi\in\arco\mid\ord(g \circ \varphi)=\ord(z^{\ex+p} \circ \varphi)= \ord (\mathbf{x}^{\mathbf{N}_I}z^p\circ \varphi)\}\\
&=
\{\varphi\in\arco\mid(\ex + p)b_{z} = \langle\mathbf{b},\mathbf{N}\rangle,\
\ac(z^{\ex+p} \circ \varphi ) \ne  \ac(\mathbf{x}^{\mathbf{N}_I}z^p  \circ \varphi)\}\\
\arco_{\rho^\ast}&:=\{\varphi\in\arco\mid\ord(g \circ \varphi)>\ord(z^{\ex+p} \circ \varphi)= \ord (\mathbf{x}^{\mathbf{N}_I}z^p\circ \varphi)\}\\
&=
\{\varphi\in\arco\mid(\ex + p)b_{z} = \langle\mathbf{b},\mathbf{N}\rangle,\
\ac(z^{\ex+p} \circ \varphi ) =  \ac(\mathbf{x}^{\mathbf{N}_I}z^p  \circ \varphi)\}
\end{align*}
The above subspaces induce a partition of each  space $\mathcal{X}_{n,m}$ of $n$-jets into four subsets
\[
\mathcal{X}_{n,m}^{\bullet} = \{ \varphi \in \mathcal{X}_{n,m} \mid\varphi \in\arco_\bullet \}, \quad  \mbox{ with } \quad \bullet \in \{ \sigma^+, \rho , \sigma^-, \rho^\ast \},
\]
and splits the (naive) zeta function into four summands
\begin{equation} \label{eq:W}
  \Znv{}(g, \omega, T)_\zero =\!\!\!\!\!\!\! \!\!\!\sum_{\bullet \in \{ \sigma^+, \rho , \sigma^-, \rho^\ast \}} \!\!\!\!\!\!\!\!\!\!\!\!\str^{\bullet}, \qquad
\text {with }  \, \str^\bullet :=\sum_{n\in \NN}\sum_{m \in \ZZ_{\geq 0}} [\mathcal{X}_{n,m}^\bullet]{\LL}^{-n(\abs{I}+1) -m}T^n.
\end{equation}
The spaces $\arco_\bullet$ suggest the following conical subdivision of $\RR^{I_z}_{\geq 0}$ which is dual to the Newton polyhedron of $g$.
\begin{align*}
\sigma^{+}\!&:=\{\mathbf{b}\in\RR^{I_z}_{\geq 0}\mid\langle\mathbf{b},\mathbf{N}\rangle
< (\ex + p)b_{z}\},\\
\sigma^{-}\!&:=\{\mathbf{b}\in\RR^{I_z}_{\geq 0}\mid\langle\mathbf{b},\mathbf{N}\rangle
> (\ex + p)b_{z}\},\\
\rho\ &:=\{\mathbf{b}\in\RR^{I_z}_{\geq 0}\mid\langle\mathbf{b},\mathbf{N}\rangle
= (\ex + p)b_{z}\}.
\end{align*}
For $\mathbf{b}\in\RR^{I_z}_{\geq 0}$ we denote
\[
m_{\mathbf{b}}:=\langle\mathbf{b},\boldsymbol{\nu}\rangle-\abs{\mathbf{b}},\text{ and }
n_{\mathbf{b}}:=
\begin{cases}
\langle\mathbf{b},\mathbf{N}\rangle&\text{ if }\mathbf{b}\in\sigma^+\\
(\ex + p)b_{z} &\text{ if }\mathbf{b}\in\sigma^-\\
 (\ex + p)b_{z}=\langle\mathbf{b},\mathbf{N}\rangle&\text{ if }\mathbf{b}\in\rho.
\end{cases}
\]
These quantities are related with the order of the pullback of $\omega$ and $g$ respectively. More precisely, if $\varphi$ is an arc in $\arco$
of order $\mathbf{b}$ and $\varphi\in\arco_\bullet$, for  $\bullet \in \{ \sigma^+, \sigma^-,\rho \}$, then $m_{\mathbf{b}}=\ord(\varphi^*\omega)$, while $n_{\mathbf{b}}=\ord(g\circ\varphi)$.
Hence,  we can write
\[
W^\bullet =  \sum_{ \mathbf{b} \in \bullet  \cap \NN^{I_z}} [\mathcal{X}_{n_{\mathbf{b}},m_{\mathbf{b}}}^\bullet]
\LL^{-n_{\mathbf{b}} (\abs{I} + 1)-m_{\mathbf{b}}}T^{n_{\mathbf{b}}}.
\]
For $\varphi\in\arco_{\rho^\ast}$, then $\mathbf{b}=\ord\varphi\in\rho$, $m_\mathbf{b}=\ord(\varphi^*\omega)$ and
$n_\mathbf{b}<\ord(g\circ\varphi)$.
Then,
\[
W^{\rho^\ast} = \sum_{ \mathbf{b} \in \rho  \cap \NN^{I_z}}
\left (\sum_{i=1}^{\infty}   [\mathcal{X}_{n_{\mathbf{b}}+i,m_{\mathbf{b}}}^{\rho^\ast}] \LL^{-i(\abs{I}+1)} T^i
\right) \LL^{-n_{\mathbf{b}}(\abs{I} + 1) - m_{\mathbf{b}}} T^{n_{\mathbf{b}}}.
\]
The following lemma can be deduced from \cite{dh:01,ACNLM-AMS,GG:14}.
The key of the proof is to study the class $[V(g_0)]$ where
$V(g_0)$ is the zero locus of $g_0$ in $(\CC^*)^{I_z}$ and the class of the complement of $V(g_0)$ in $(\CC^*)^{I_z}$.
Let us denote $N_I:=\gcd \mathbf{N}_I$ and $e_I:=\gcd(\ex,N_I)$. Since $V(g_0)$ is in $(\CC^*)^{I_z}$,
we can make a toric change of variables such that
\[
V(g_0)\cong(\CC^*)^{\abs{I}-1}\times\{(y,z)\in(\CC^*)^2\mid z^\ex - y^{N_I}=0\}\cong e_I\text{ copies of }(\CC^*)^{\abs{I}}.
\]

\begin{lemma}\label{lema:b}
Let $\mathbf{b}\in\NN^{I_z}$.
\begin{enumerate}[label=\rm(\arabic{enumi})]
\item If $\mathbf{b}\in\sigma^\pm$, then
$[\mathcal{X}_{n_\mathbf{b},m_\mathbf{b}}^{\sigma^\pm}] \LL^{-n_{\mathbf{b}}(\abs{I} + 1) - m_{\mathbf{b}}}=
(\LL-1)^{\abs{I} + 1}\LL^{-\langle\mathbf{b},\boldsymbol{\nu}\rangle}$.
\item If $\mathbf{b}\in\rho$, then
\[
[\mathcal{X}_{n_\mathbf{b}, m_\mathbf{b}}^{\rho}]\LL^{-n_{\mathbf{b}}(\abs{I} + 1) - m_{\mathbf{b}}} =
(\LL-1)^{\abs{I}}(\LL-1 - e_I)\LL^{- \langle\mathbf{b},\boldsymbol{\nu}\rangle}.
\]
\item If $\mathbf{b}\in\rho$ and $i>0$, then
\[
[\mathcal{X}_{n_\mathbf{b} + i, m_\mathbf{b}}^{\rho^\ast}] \LL^{-i(\abs{I}+1)-n_{\mathbf{b}}(\abs{I} + 1) - m_{\mathbf{b}}}=
e_I(\LL - 1)^{\abs{I}+1}\LL^{-i -\langle\mathbf{b},\boldsymbol{\nu}\rangle}.
\]
\end{enumerate}
\end{lemma}

\begin{remark}\label{Wmot}
For
$\Zmot{}(g,\omega, T)_\zero$ the computation is similar, see~\cite{gui:02,GG:14, bn:20}, and one obtains an expression
\[
\Zmot{}(g, \omega, T)_\zero = \widetilde{\str}^{\sigma^+}+ \widetilde{\str}^{\sigma^-} + \widetilde{\str}^\rho + \widetilde{\str}^{\rho^\ast},
\]
where $\widetilde{\str}^\bullet$ is defined analogously as in the naive case,
 and
\begin{enumerate}
\item if  $\mathbf{b} \in \sigma^+$ then   $[\mathcal{X}^{1,\sigma^+}_{n_\mathbf{b}, m_\mathbf{b}}] = [\mu_{\gcd(N_I, \ex+p)}](\LL-1)^{\abs{I}}\LL^{n_{\mathbf{b}}(\abs{I} + 1) - \abs{\mathbf{b}}}$;
\item if  $\mathbf{b} \in \sigma^-$ then   $[\mathcal{X}^{1,\sigma^-}_{n_\mathbf{b}, m_\mathbf{b}}] = [\mu_{\ex+p}](\LL-1)^{\abs{I} }\LL^{n_{\mathbf{b}}(\abs{I} + 1) - \abs{\mathbf{b}}}$;
\item if  $\mathbf{b} \in \rho$  then
\begin{align*}
[\mathcal{X}^{1,\rho}_{n_\mathbf{b}, m_\mathbf{b}}] &=
[V(g_1)]
\LL^{n_{\mathbf{b}}(\abs{I} + 1) - \abs{\mathbf{b}}}\\
[\mathcal{X}^{1,\rho^\ast}_{n_\mathbf{b}, m_\mathbf{b}}] &= e_I (\LL-1)^{|I|} \LL^{i\abs{I}}\text{ for } i>0.
\end{align*}
where $V(g_1)$ is the zero locus of $g_0-1$ in $(\CC^*)^{I_z}$ with a good action of $\hat{\mu}$. Note that after a toric change of variables it satisfies
\[
V(g_1)\cong(\CC^*)^{\abs{I}-1}\times\{(y,z)\in(\CC^*)^2\mid z^\ex - y^{N_I}=1\}.
\]
\end{enumerate}
\end{remark}

According to a celebrated result by Denef and Loeser  \cite[Theorem 2.2.1]{Denef-LoeserIgusa} or \cite[Corollary~3.3.2]{Denef-LoeserBarca}, each of the previous partial zeta functions can be expressed as a rational function, i.e., an element in the $\mathcal{M}_\CC$-submodule of $\mathcal{M}_\CC[[T ]]$
generated by 1 and by finite products of terms of the form $\LL^j T^i (1- \LL^j
T^i)^{-1}$, with $j \in \ZZ$ and $i \in \ZZ_{>0}$. Using the machinery of integer points on rational polyhedral cones from \cite[Section 4.6]{stanley}, it is possible to get such rational expressions.

Given a pointed convex rational polyhedral cone $\mathcal{C} \subset \RR^n$, let $\mathring{\mathcal{C}}$ be its interior. We define the generating function of  $\mathcal{C}$ as
\[
\Phi_\mathcal{C}(x) = \sum_{\boldsymbol{\alpha} \in
\mathring{\mathcal{C}}
} \mathbf{x}^{\boldsymbol{\alpha}}.
\]
Let $\ba_i$, $i=1,\dots,\dim\mathcal{C}$, be
the primitive quasi-generators of a simplicial cone $\mathcal{C}$. The set
\[
D_\mathcal{C}:=\left\{ \boldsymbol{\lambda} \in\mathring{\mathcal{C}}\, \middle|\ \boldsymbol{\lambda} = \sum_{i=1}^{\dim\mathcal{C}} \lambda_i \ba_i \quad \text{with}\ \lambda_i\in(0,1]
 \right\},
\]
is finite. Then,
\[
\Phi_\mathcal{C}(\mathbf{x})=P_{\mathcal{C}}(\mathbf{x})\prod_{i=1}^{\dim\mathcal{C}} (1 - \mathbf{x}^{\ba_i})^{-1}\text{ where }P_{\mathcal{C}}(\mathbf{x}):=\sum_{\boldsymbol{\beta} \in D_\mathcal{C}}\mathbf{x}^{\boldsymbol{\beta}}.
\]

\begin{remark}\label{lema:conos}
The cones $\sigma^+$ and $\rho$ are simplicial.
A system of quasi-generators of $\rho$ is given by
\[
\mathbf{v}_i:=\frac{\ex\mathbf{e}_i + N_i\mathbf{e}_z}{e_i},
i\in I,\quad
e_i:=\gcd(\ex,N_i);
\]
adding $\mathbf{e}_z$ we obtain a system of quasi generators of $\sigma^+$.
Then,
\[
\Phi_{\rho}(\mathbf{x}) = P_\rho(\mathbf{x})
\prod_{k\in I} P_k(x_k, z)^{-1},\qquad
\Phi_{\sigma^+}(x) = P_{\sigma^+}(\mathbf{x})
P_z(z)^{-1}\prod_{k\in I} P_k(x_k,z)^{-1},
\]
where
\[
P_k(x_k, z):=1 - \left(x_k^{\ex}z^{N_k}\right)^{\frac{1}{e_k}},\ k\in I,\quad P_z(z)=1-z.
\]
In general, the cone $\sigma^-$ is not simplicial but we have
\begin{equation*}
\Phi_{\sigma^-}=\Phi_{\RR^{I_z}_{\geq 0}} -\Phi_{\sigma^+} -\Phi_{\rho},\qquad
\Phi_{\RR^{I_z}_{\geq 0}}=(1-z)^{-1}
\prod_{i\in I} (1 - x_i)^{-1}.
\end{equation*}
\end{remark}

For  $\ba,\mathbf{b}\in\ZZ_{\geq 0}^{I}$, we denote
$\LL^{\ba} T^{\mathbf{b}}:=(\LL^{-{a}_k}T^{{b}_k}\mid k\in I)$.
We may replace $I$ by $I_z$ if necessary.

Then next proposition expresses $\str^{\bullet}$ in terms of the generating functions. Beware the generating functions are factorized in order to simplify the Euler specialization in \S \ref{sec:top_zeta_binomials}.

\begin{prop}\label{prop:W}
Using the generating functions of the cones from Remark{\rm~\ref{lema:conos}}, the terms $\str^{\bullet}$ from \eqref{eq:W} are expressed as follows:
\begin{align*}
\str^{\sigma^+\!\!}{}&=
P_{\sigma^+}\left(\mathbf{x}=\LL^{-\boldsymbol{\nu}}T^{\mathbf{N}}\right)
K(\LL, T)
H(\LL, T),
\\
\str^{\sigma^-\!\!}{}&=
\tilde{K}(\LL, T)
\prod_{k\in I}\frac{\LL-1}{1-\LL^{-\nu_k}}-P_{\sigma^+}(\mathbf{x}=\LL^{-\boldsymbol{\nu}}T^{(\ex + p)\mathbf{e}_z})
\tilde{K}(\LL, T)
\tilde{H}(\LL, T)\\
&-(\LL-1)
P_{\rho}(\mathbf{x}=\LL^{-\boldsymbol{\nu}}T^{(\ex + p)\mathbf{e}_z})
\tilde{H}(\LL, T),
\\
\str^{\rho}&=
(\LL - 1 - e_I)P_{\rho}\left(\mathbf{x}=
\LL^{-\boldsymbol{\nu}}T^{\mathbf{N}}\right)
H(\LL, T),
\\
\str^{\rho^\ast\!\!}{}&=
e_I\LL^{-1} T
 P_{\rho}\left(\mathbf{x}=
\LL^{-\boldsymbol{\nu}}T^{\mathbf{N}}\right)
\frac{\LL - 1}{1-\LL^{-1} T}
H(\LL, T),
\end{align*}
where
\[
H(\LL, T):=\prod_{k\in I}\frac{\LL - 1}{P_k\left(\LL^{-\nu_k} T^{N_k}, \LL^{-\nu_z} T^{p}\right)},\quad
\tilde{H}(\LL, T):=\prod_{k\in I}\frac{\LL-1}{P_k(\LL^{-\nu_k},\LL^{-\nu_z}T^{\ex + p})},
\]
and
\[
K(\LL, T):=\frac{\LL - 1}{1-\LL^{-\nu_z} T^{p}},\qquad
\tilde{K}(\LL, T):=\frac{\LL - 1}{1-\LL^{-\nu_z} T^{\ex + p}}.
\]
\end{prop}

\begin{proof}
We combine the definition of $\str^\bullet$, with Lemma~\ref{lema:b} and the generating functions.
Let us start with $\sigma^+$:
\[
\str^{\sigma^+} =  (\LL-1)^{\abs{I} + 1}\sum_{ \mathbf{b} \in \sigma^+  \cap \NN^{I_z}} \LL^{-\langle\mathbf{b},\boldsymbol{\nu}\rangle}T^{\langle\mathbf{b},\mathbf{N}\rangle} = (\LL-1)^{\abs{I} + 1}\Phi_{\sigma^+}(\mathbf{x}=\LL^{-\boldsymbol{\nu}}T^{\mathbf{N}}).
\]
Using the expression for $\Phi_{\sigma^+}$, we obtain the result in the statement.
For $\sigma^-$ we have
\begin{align*}
\str^{\sigma^-} = (\LL-1)^{\abs{I} + 1} \sum_{ \mathbf{b} \in \sigma^-  \cap \NN^{I_z}}
\LL^{- \langle\mathbf{b},\boldsymbol{\nu}\rangle}
T^{(\ex + p)b_{z}}
=
(\LL-1)^{\abs{I} + 1}\Phi_{\sigma^-}(\mathbf{x}=\LL^{-\boldsymbol{\nu}}T^{(\ex + p)\mathbf{e}_z}).
\end{align*}
and the statement follows.
For $\rho$ we have:
\[
\str^\rho = (\LL-1)^{\abs{I}}(\LL-1 - e_I) \sum_{ \mathbf{b} \in \rho  \cap \NN^{I_z}} \LL^{- \langle\mathbf{b},\boldsymbol{\nu}\rangle}T^{\langle\mathbf{b},\mathbf{N}\rangle}=
 (\LL-1)^{\abs{I}}(\LL-1 - e_I) \Phi_{\rho}(\mathbf{x}=\LL^{-\boldsymbol{\nu}}T^{\mathbf{N}}).
\]
Finally, for $\rho^\ast$:
\[
\str^{\rho^\ast} = e_I(\LL - 1)^{\abs{I}+1}\sum_{i=1}^{\infty}  \LL^{-i} T^i\cdot\sum_{ \mathbf{b} \in \rho  \cap \NN^{I_z}}
 \LL^{-\langle\mathbf{b},\boldsymbol{\nu}\rangle} T^{\langle\mathbf{b},\mathbf{N}\rangle}\!=\!
e_I(\LL - 1)^{\abs{I}+1}
\frac{\LL^{-1} T}{1-\LL^{-1} T}\Phi_{\rho}(\mathbf{x}=\LL^{-\boldsymbol{\nu}}T^{\mathbf{N}}),
\]
and the result follows.
\end{proof}

\begin{remark} For the summands of $\Zmot{}(g,\omega, T)_\zero$  one obtains
\[
(\LL-1 )\widetilde{\str}^{\sigma^+}=[\mu_{\gcd(N_I, \ex-p)}]{\str}^{\sigma^+}, \quad (\LL-1 )\widetilde{\str}^{\sigma^-}=[\mu_{\ex+p}]{\str}^{\sigma^+}, \quad (\LL-1 )\widetilde{\str}^{\rho}={\str}^{\rho}
\]
and
\[
\widetilde{\str}^{\rho^\ast} =[ \{(y,z)\in(\CC^*)^2\mid z^\ex - y^{N_I}=1\} ] \cdot P_{\rho}\left(\mathbf{x}=
\LL^{-\boldsymbol{\nu}}T^{\mathbf{N}}\right)
\frac{\LL - 1}{1-\LL^{-1} T}
H(\LL, T).
\]
\end{remark}

\section{Topological Zeta functions for certain binomials}
\label{sec:top_zeta_binomials}

This section is devoted to compute the Euler specialization of the naive and motivic zeta functions computed in \S\ref{sec:motivic_zeta_binomials}.

Recall that the topological zeta function
is obtained from the naive zeta function by specializing with the
Euler characteristic, see \eqref{eq:zetatop}. We define  $\str^{\bullet}_{\topo}$ as $\chi_{\topo} (\str^{\bullet}(T=\LL^{-s}))$
for the terms $\str^{\bullet}$ of \eqref{eq:W}. Notice that
\begin{equation*}
\Ztop{}(g,\omega,s)_\zero= \str^{\sigma^+}_{\topo} + \str^{\sigma^-}_{\topo} + \str^{\rho}_{\topo} + \str^{\rho^\ast}_{\topo}
\end{equation*}
\begin{lemma}\label{aportop}
Consider $g(\mathbf{\boldsymbol{x}},z)= z^p (z^\ex - \mathbf{x}^{\mathbf{N}_I})$, and
$\omega = \mathbf{\boldsymbol{x}}^{\boldsymbol{\nu}} z^{\nu_z} \frac{\dif \mathbf{\boldsymbol{x}}}{\mathbf{x}} \frac{\dif z}{z}$ as in \eqref{eq:g} and \eqref{eq:omega}.
The  terms   $\str^{\bullet}_{\topo}$ of the topological
zeta function $\Ztop{}(g,\omega,s)_\zero$ are expressed in terms of $r :=\frac{(\ex+p)s+ \nu_z}{\ex}$ as follows:
\begin{align*}
\str^{\sigma^+}_{\topo}&=
\frac{1}{\nu_z + p s}
\prod_{k\in I}
\frac{1}{
N_k r + {\nu}_k},\\
\str^{\sigma^-}_{\topo}&=
\frac{1}{\ex r}
\left(\prod_{k\in I} \frac{1}{ {\nu}_k} -
\prod_{k\in I}
\frac{1}{N_k r + {\nu}_k} \right),
\\
\str^{\rho}_{\topo}&=
\frac{-  e_I^2 }{\ex}
\prod_{k\in I} \frac{1}{N_k r + {\nu}_k},
\\
\str^{\rho^\ast}_{\topo}&=
\frac{1}{s+1}\termino{\rho}{\topo}{g,\omega}.
\end{align*}
\end{lemma}

\begin{remark}\label{rt}Notice that if $p=0$ and $\nu_z=1$ then $r$ becomes $t := s + \frac{1}{\ex}$.\end{remark}

\begin{proof} Recall that the image of the  topological Euler specialization morphism $\chi: {\mathcal{M}}_\CC \rightarrow \ZZ$ can be recovered from the following two facts. First, since  $\chi$ is additive, $\chi(\LL)=1$. Note that for $a\in\NN$,
\[
[\PP^a]= 1 + \LL + \LL^2 + \cdots + \LL^a=
\frac{1-\LL^{a+1}}{1-\LL}\in \KVarC{}\Longrightarrow\chi([\PP^a])= a +1.
\]
Secondly, if we allow
\emph{formal dimensions}
we have that
\[
\chi \left (  \frac{1 - \LL}{ 1- \LL^{a+bs}} \right ) = \chi \left ( \frac{1}{[\PP^{a + bs -1}]} \right ) =\frac{1}{a+bs}.
\]
As a consequence,
\[
\chi(K(\LL, T))=\frac{1}{\nu_z + p s},\quad
\chi(\tilde{K}(\LL, T))=\frac{1}{\nu_z + (\ex + p) s}=\frac{\ex}{r}.
\]
On the other side, the Euler characteristics of the terms $P_\bullet$ appearing in
Proposition~\ref{prop:W} coincide with $\abs{D_\bullet}$ (the \emph{multiplicity} of the cone).
In Remark~\ref{lema:conos} we described the primitive
quasi-generators of the $\abs{I}+1$-dimensional simplicial cone $\sigma^+$ and
the $\abs{I}$-dimensional simplicial cone $\rho$.
The multiplicity of $\sigma^+_I$ equals the absolute
value of the determinant of the matrix formed by these
quasi-generators, i.e.,
\begin{equation*}
\ex^{\abs{I}}\displaystyle\prod_{k\in I} e_k^{-1}.
\end{equation*}
The multiplicity of  $\rho$ it is the $\gcd$ of the $\abs{I}$-minors of the matrix of the generators
which turns out to be
\begin{equation*}
\ex^{\abs{I}-1} e_I\displaystyle\prod_{k\in I} e_{k}^{-1}
\end{equation*}
The Euler characteristic of the factor of $\tilde{H}(\LL, T)$ for $k\in I$ is
\[
\chi\left(\frac{\LL-1}{1-\LL^{-\frac{(\ex\nu_k + \nu_z N_k)+s(\ex + p) N_k}{e_k}}}\right)=
\frac{e_k}{(\ex\nu_k + \nu_z  N _k)+s(\ex + p) N_k}=\frac{e_k}{\ex (N_kr + \nu_k)}
\]
and it coincides with the Euler characteristic of the corresponding factor in $H(\LL, T)$.
Combining all these facts the result is proved.
\end{proof}

Recall also that the local twisted topological zeta function
is obtained from the motivic one by specializing with the equivariant Euler characteristic, see~\eqref{eq:zetatop}.
As before, we define  the terms $\widetilde{\str}^{\bullet}_{\topo}$ as $\chi_{\topo} ((\LL-1) \widetilde{\str}^{\bullet}(T=\LL^{-s}), \alpha)$, with $\alpha$ a character of order $l$,
for the terms $\widetilde{\str}^{\bullet}$ introduced in Remark \ref{Wmot}. Notice that
\begin{equation}\label{eq:tzetatopW}
\Ztop{(l)}(g,\omega,s)_\zero= \widetilde{\str}^{\sigma^+}_{\topo} + \widetilde{\str}^{\sigma^-}_{\topo} + \widetilde{\str}^{\rho}_{\topo} + \widetilde{\str}^{\rho^\ast}_{\topo}
\end{equation}
The following proposition is needed to compute the equivariant Euler  specializations.
For $\bullet\in\{\sigma^+, \sigma^-, \rho\}$ let us define
\[
\mathcal{N}(\bullet):=\gcd
\left\{
\min\left(\langle \ba , (\ex + p) \mathbf{e}_{z}\rangle , \left\langle \ba, \mathbf{N}
\right\rangle  \right)
\middle| \,  \ba \in 	\bullet\cap\ZZ_{> 0}^{I_z} \right\}.
\]

\begin{prop}\label{prop:Nbullet}
	The numbers $ \mathcal{N}(\bullet)$ are given by
\[
\mathcal{N}(\bullet) =
\begin{cases}
\gcd\mathbf{N} & \text{if } \bullet =  \sigma^+\\
\ex + p& \text{if } \bullet= \sigma^-\\
\lcm (\gcd\mathbf{N}_I, \ex)\dfrac{\ex + p}{\ex} & \text{if } \bullet= \rho.
\end{cases}
\]
\end{prop}

\begin{proof}
Let us start with $\sigma^-$. By definition if $\ba \in 	\sigma^-\cap\ZZ_{\geq 0}^{I_z}$,
then
$\langle\ba,\mathbf{N}\rangle
> (\ex + p)a_{z} =\langle \ba , (\ex + p) \mathbf{e}_{z}\rangle$ and then
$\ex + p$ divides $\mathcal{N}(\sigma^-)$.
We have
\[
\ba':=\frac{1}{\abs{\mathbf{N}_I}}\sum_{i\in I} e_i\mathbf{v}_i=
\frac{\ex}{\abs{\mathbf{N}_I}}\sum_{i\in I}\mathbf{e}_i + \mathbf{e}_z,\qquad
\langle\ba',\mathbf{N}\rangle=\langle\ba',(\ex + p)\mathbf{e}_z\rangle=\ex + p.
\]
Let
\[
\ba'':=\left(1+\left\lceil\frac{\ex}{\abs{\mathbf{N}_I}}\right\rceil-\frac{\ex}{\abs{\mathbf{N}_I}}\right)\sum_{i\in I} \mathbf{e}_i,
\qquad \langle\ba'',\mathbf{N}\rangle>0,\quad \langle\ba'',(\ex + p)\mathbf{e}_z\rangle=0.
\]
Then
\[
\ba:=\ba'+\ba''\cap \sigma^- \in\ZZ^I_{>0},\qquad \langle\ba,(\ex + p)\mathbf{e}_z\rangle=\ex + p,
\]
and we deduce that $\mathcal{N}(\sigma^-)=\ex + p$.

Let us continue with $\sigma^+$. By definition if $\ba \in 	\sigma^+\cap\ZZ_{\geq 0}^{I_z}$,
then $\left\langle \ba, \mathbf{N} \right\rangle <\langle \ba , (\ex + p) \mathbf{e}_{z}\rangle$. Since
the coefficients of $\ba$ are integers, then  $\gcd\mathbf{N}$ divides the minimum
and hence also $\mathcal{N}(\sigma^+)$.
Fix $i\in I$, we consider
\[
\ba_1=\mathbf{e}_i+N_i\sum_{j\in I\setminus\{i\}} \mathbf{e}_j + N_i M e_z\in\ZZ_{>0}^{I_z}, \quad
\ba_2=2\mathbf{e}_i+N_i\sum_{j\in I\setminus\{i\}} \mathbf{e}_j + N_i M e_z\in\ZZ_{>0}^{I_z}, \quad
M\gg 0.
\]
Note that $\ba_1,\ba_2\in\sigma^+$ and
\[
\langle\ba_k,\mathbf{N}\rangle=N_i\left(k+\sum_{j\in I\setminus\{i\}} N_j +M p\right)\Longrightarrow
\gcd(\langle\ba_1,\mathbf{N}\rangle,\langle\ba_2,\mathbf{N}\rangle)=N_i.
\]
Then $\mathcal{N}(\sigma^+)$ divides $N_i$, for $i\in I$. Let
\[
\ba_M=p\sum_{j\in I} \mathbf{e}_j + M e_z\in\ZZ_{>0}^{I_z}, \quad
M\gg 0.
\]
We have that $\ba_M\in\sigma^+$ and
$\langle\ba_M,\mathbf{N}\rangle=p(\abs{\mathbf{N}_I} + M)$ and then $\mathcal{N}(\sigma^+)$ divides $pM$
for $M\gg 0$ and then it divides~$p$. We deduce that $\mathcal{N}(\sigma^+)=\gcd\mathbf{N}$.

Finally for $\mathcal{N}(\rho)$, if $\ba\in\rho\cap\NN^{I_z}$, we have that
\[
\sum_{k\in I}a_k N_k + a_z p =a_z(\ex + p)\Longleftrightarrow
\sum_{k\in I}a_k N_k = a_z \ex \Longrightarrow
\mathcal{N}(\rho) \text{ divides } \left(\sum_{k\in I}a_k N_k\right)\frac{\ex + p}{\ex},
\]
and then $\mathcal{N}(\rho)$ divides $\lcm(\gcd\mathbf{N}_I,\ex)\frac{\ex + p}{\ex}$.
Let us fix $i\in I$ and for $k=1,2$, let
\[
\ba_k=\frac{\lcm(\ex,N_i)}{N_i}\left(k\mathbf{e}_i+N_i\sum_{j\in I\setminus\{i\}} \mathbf{e}_j\right)+
\frac{\lcm(\ex,N_i)}{\ex}\left(k+\sum_{j\in I\setminus\{i\}} N_j\right) e_z\in\ZZ_{>0}^{I_z}.
\]
These vectors also verify that $\ba_k\in\rho$ and the common inner product is
\[
\lcm(\ex,N_i)\frac{\ex + p}{p}\left(k+\sum_{j\in I\setminus\{i\}} N_j\right);
\]
their $\gcd$ is $\lcm(\ex,N_i)\frac{\ex + p}{p}$. The $\gcd$ of all these numbers
is the number of the statement and the result follows.
\end{proof}

\begin{cor}\label{twistedaportop} Consider
$g_0(\mathbf{x},z) =  z^\ex - \mathbf{x}^{\mathbf{N}_I}$
and $\omega := \mathbf{x}^{\mathbf{\boldsymbol{\nu}}} z \frac{\dif \mathbf{x}}{\mathbf{x}} \frac{\dif z}{z}$, i.e., as in \eqref{eq:g} and \eqref{eq:omega} with $p=0$ and $\nu_z=1$.
The  terms   $\widetilde{\str}^{\bullet}_{\topo}$ of the topological
zeta function $\Ztop{(e)}(g_0,\omega,s)_\zero$ are expressed in terms of $t :=s + \frac{1}{\ex}$  as follows.
\begin{itemize}
\item $\widetilde{\str}^{\sigma^+}_{\topo} = {\str}^{\sigma^+}_{\topo}$ if $e \mid N_I$ and $0$ otherwise.
\item $\widetilde{\str}^{\sigma^-}_{\topo} = {\str}^{\sigma^-}_{\topo}$ if $e \mid \ex$ and $0$ otherwise.
\item $\widetilde{\str}^{\rho}_{\topo} = {\str}^{\rho}_{\topo}$ if  $e \mid \lcm(\ex, N_I)$ and $0$ otherwise.
\item $\widetilde{\str}^{\rho^\ast}_{\topo} = 0$.
\end{itemize}
\end{cor}
\begin{proof}
	The proof follows the same lines of the proof of Lemma~\ref{aportop}, taking into account the definition of $\widetilde{\str}^{\bullet}_{\topo}$ and Proposition~\ref{prop:Nbullet}.
\end{proof}

\section{Computation of the local motivic zeta functions of \texorpdfstring{$G$}{G}}
\label{sec:motivic_zeta_susp}

In this section we compute the local naive and motivic zeta functions of $G$ from \S\ref{subsec:general}. For this we
consider two fixed monomial volume forms
\[
\omega_d:=
x_1^{\nu_1}\cdot\ldots\cdot x_d^{\nu_d}
\frac{\dif x_1}{x_1}\dots \frac{\dif x_d}{x_d},\quad \text{and}\quad
\omega_{d+1}:=\omega_dz^{\nu_{z}}
\frac{\dif z}{z},\quad \text{with }
\nu_1,\dots,\nu_d,\nu_z\in\NN.
\]
In order to compute  $\Zmot{}(G,\omega_{d+1}, T)_\zero$ and $\Znv{}(G,\omega_{d+1}, T)_\zero$  we use the following double partition of $\mathcal{L}(\CC^{d+1})_\zero$. In first place we consider the contact of the arcs with the functions $z^\ex$, $f$, and $G$. Then for each $\mathcal{X}_{n,m}$ we define the following subsets:
\begin{align*}
\mathcal{X}_{n,m}^{\sigma^+}&:=
\{\varphi \in \mathcal{X}_{n,m}\mid (\ex+p) \ord_t \varphi_z > \ord_t (z^pf)\circ \varphi= n\}\\
\mathcal{X}_{n,m}^{\sigma^-}&:=
\{\varphi \in \mathcal{X}_{n,m}\mid (\ex+p) \ord_t \varphi_z < \ord_t (z^pf)\circ \varphi = n\}\\
\mathcal{X}_{n,m}^{\rho\hphantom{{}^+}}&:=
\{\varphi \in \mathcal{X}_{n,m}\mid (\ex+p) \ord_t \varphi_z = \ord_t (z^pf)\circ \varphi = n\}\\
\mathcal{X}_{n,m}^{\rho^\ast}\ &:=
\{\varphi \in \mathcal{X}_{n,m}\mid (\ex+p) \ord_t \varphi_z = \ord_t (z^pf)\circ \varphi < n\}.
\end{align*}
These subsets form a partition of $\mathcal{X}_{n,m}$ and induce a decomposition of $\Znv{}(G,\omega_{d+1},T)_\zero$ into four pieces $\Znv{\bullet}(G,\omega_{d+1},T)_\zero$
as in \eqref{eq:zetamot}, i.e.
\begin{equation}\label{eq:globalbullet}
\Znv{}(G,\omega_{d+1}, T)_\zero=\sum_{\bullet\in \{\sigma^+,\sigma^-,\rho,\rho^\ast\}} \Znv{\bullet}(G,\omega_{d+1},T)_\zero.
\end{equation}
Analogously one can make a partition of  $\mathcal{X}^{1}_{n,m}$ and induce a decomposition of $\Zmot{}(G,\omega_{d+1},T)_\zero$.

The \emph{second partition} just takes into account the first $d$ coordinates and depends on the choice of an embedded  resolution $\pi : Y \rightarrow \CC^d$ of $f\omega_d$ as in \S\ref{subsec:dlzf}. The map $\pi\times\id:Y\times\CC\to\CC^{d+1}$ is proper
and birational. Note that $\pi^{-1}(\zero)$ and $(\pi\times\id)^{-1}(\zero)$ can be canonically identified, in particular recall that the irreducible components of $\pi^* f\omega_d$ are $\{E_j\mid j\in J\}$.

For $I\subset J$ and a point $o_I \in E_I^\circ$ we consider a system of coordinates~$\mathcal{U}$ as in \S\ref{subsec:dlzf}
with the two subsets of coordinates; $\mathbf{x}_I$ and $\mathbf{y}$. In
$\mathcal{U}\times\CC$ we add the $z$-variable.
Under this condition,
the pullbacks of $G$ and $\omega_{d+1}$ under $\pi\times\id$ will be denoted by $G_{I_z}$ and $\omega_{I_z}$, they have local equations $z^p(z^\ex-\mathbf{x}_I^{N_I} u(\mathbf{x}, \mathbf{y}))$ and $v\mathbf{x}_I^{\mathbf{\nu}_I}  z^{\nu_{z}}  \frac{\dif \mathbf{x}_I}{\mathbf{x}_I} \dif \mathbf{y} \frac{\dif z}{z}$
for some non-vanishing functions $u, v$ which may depend on all the variables. Note that the vector $\mathbf{\nu}_I$ depends only on the index set $I$ and the initial form $\omega_{d}$, while $\nu_{z}$ is independent of $I$. Since $\ord G_I \circ \varphi$ and $\ord \omega_I \circ \varphi$ are independent of $\ord \varphi_{\mathbf{y}}$, the subset coordinates $\mathbf{y}$ can be omitted and the computation of the contributions of the different strata reduces to \S \ref{sec:motivic_zeta_binomials} and \S \ref{sec:top_zeta_binomials}.

Combining both partitions we arrive at
\begin{equation}\label{eq:doublepart}
\Znv{\bullet}(G,\omega_{d+1},T)_\zero=
\sum_{I \subset J}
\str^{\bullet}_{I}\quad\text{and}\quad  \Zmot{\bullet}(G,\omega_{d+1},T)_\zero=
\sum_{I \subset J}
\widetilde{\str}^{\bullet}_{I}.
\end{equation}

\begin{lemma}\label{lema:globalW}
The  summands $\str^{\bullet}_{I}$ of the naive zeta function can be expressed as
\[\str_I^\bullet = [E_I^\circ]\str^\bullet,
\]
where $\str^\bullet$ are given in Proposition{\rm~\ref{prop:W}}. A similar expression can be given for the motivic zeta function.
\end{lemma}
\begin{proof}This is consequence of the local triviality of the local forms $G_I$ along the stratum $E^\circ_i$ and  the Stratification principle.
\end{proof}

\section{A general formula for the topological zeta functions of \texorpdfstring{$G$}{G}
and \texorpdfstring{$F$}{F}}
\label{sec:MainResults}

We are going to use the result of the previous section to compute the topological zeta function of~$G$
with respect to a form. As it was stated in the Introduction we are mostly interested
in a form where $\nu_i=1$ for $i\in\{1,\dots,d\}$ and $\nu_z=d + 1$. Since
for general monomial forms the arguments follow the same guidelines, we state the theorem
in a more general setting.

\begin{theorem} \label{thmpsuspACNLM}
Consider $G(\mathbf{\boldsymbol{x}},z)= z^p (z^\ex - f(\mathbf{x}))$, and
$\omega_{d+1} = \mathbf{\boldsymbol{x}}^{\boldsymbol{\nu}} z^{\nu_z} \frac{\dif \mathbf{\boldsymbol{x}}}{\mathbf{x}} \frac{\dif z}{z}$.
The following explicit expressions for the local topological zeta function in terms of  $r =\frac{(\ex+p)s+ \nu_z}{\ex}$ holds.
\begin{equation*}
\Ztop{}(G,\omega_{d+1},s)_\zero=
\frac{1}{ \ex r}+
U(r,s)
\Ztop{}(f, \omega_d, r)_\zero
	- \frac{s}{s+1} \sum_{1\neq e | \ex } \frac{J_2(e)}{\ex } \Ztop{(e)}(f, \omega_d, r)_\zero,
\end{equation*}
where $U(r,s) =
\frac{s}{r (\nu_z + ps)}
- \frac{s}{s+1}\frac{1}{\ex}$.
\end{theorem}

\begin{proof}
According to \eqref{eq:globalbullet} and \eqref{eq:doublepart}  and  Lemmata~\ref{lema:globalW} and \ref{aportop},  the topological zeta function $\Ztop{}(G,\omega_{d+1},s)_\zero$  equals
\begin{gather*}
\sum_{\emptyset \ne I \subset J}
\left(  \frac{\chi(E_I^\circ)  }{\displaystyle \ex r \prod_{i \in I} {\nu}_i}
+
\left(\frac{1}{\nu_z + ps}-\frac{1}{\displaystyle \ex r}-\frac{e_I^2}{\ex }\frac{s}{s+1}\right)
\frac{\chi(E_I^\circ)}{\displaystyle \prod_{i \in I} (N_i r + {\nu}_i)}
\right)\\
=
\sum_{\emptyset \ne I \subset J} \left(
  \frac{\chi(E_I^\circ)  }{\displaystyle \ex r \prod_{i \in I} \nu_i}
+  \frac{\ex r - (\nu_z + ps)}{\displaystyle (\nu_z + ps) \ex r} \frac{\chi(E_I^\circ)}{
\displaystyle \prod_{i \in I} (N_i r + \nu_i)} \right)-  \frac{s}{\ex (s+1)} \sum_{\emptyset \ne I \subset J}
\frac{e_I^2\chi(E_I^\circ)}{
\displaystyle  \prod_{i \in I} (N_i r + \nu_i)}.
\end{gather*}
Then, using \eqref{DLz(0)=1} and the expression of the local zeta function in terms of an embedded resolution \eqref{eq:zetas_res}, we have
\begin{gather*}
  \frac{1}{\displaystyle \ex r}
 +
 \frac{\ex r - (\nu_z + ps)}{\displaystyle (\nu_z + ps) \ex r}  \Ztop{}(f,\omega_d, r)_\zero
-  \frac{s}{\ex (s+1)} \sum_{\emptyset \ne I \subset J}
\frac{e_I^2\chi(E_I^\circ)}{
\displaystyle  \prod_{i \in I} (N_i r + \nu_i)}.
\end{gather*}
We rewrite the last factor of the last term and apply \eqref{jk}:
\begin{gather*}
\sum_{k \mid \ex }  k^2
\sum_{\substack{\emptyset \ne I \subset J\\ e_I=k }}     \frac{\chi(E_I^\circ)}{\displaystyle\prod_{i\in I} (N_i r + \nu_i)}=
\sum_{k \mid \ex }  \sum_{e \mid k } J_2(e)
\sum_{\substack{\emptyset \ne I \subset J\\ e_I = k}}     \frac{\chi(E_I^\circ)}{\displaystyle\prod_{i\in I} (N_i r+ \nu_i)}\\
= \sum_{e \mid \ex }  J_2(e)   \sum_{e \mid k \mid \ex  }
\sum_{\substack{\emptyset \ne I \subset J\\ e_I = k }}     \frac{\chi(E_I^\circ)}{\displaystyle\prod_{i\in I} (N_i r + \nu_i)}
=\sum_{e \mid \ex }  J_2(e)
\sum_{\substack{\emptyset \ne I \subset J\\ e \mid N_I }}     \frac{\chi(E_I^\circ)}{\displaystyle\prod_{i\in I} (N_i r + \nu_i)}\\
=\sum_{e \mid \ex }  J_2(e) \Ztop{(e)}(f, \omega_d, r).
\qedhere
\end{gather*}
\end{proof}

Let $\pol(f,\dif x_1\wedge\dots\wedge\dif x_d)$ be the set of poles $\Ztop{}(f,\dif x_1\wedge\dots\wedge\dif x_d,s)$. This set can be used to
bound the set of poles of $\Ztop{}(G,z^d\dif x_1\wedge\dots\wedge\dif x_d\wedge\dif z,s)$.

\begin{cor}
The set $\pol(G,z^d\dif x_1\wedge\dots\wedge\dif x_d\wedge\dif z)$ is contained in
\[
\left\{
-1, -\frac{d+1}{\ex+p}, -\frac{d+1}{p}
\right\}
\cup
\left\{
\frac{\rho_0\ex-d - 1}{\ex+p}
\middle|
\rho_0\in\pol(f,\omega_d)
\right\}.
\]
\end{cor}

As we stated in the Introduction this formula is useful for the computation of the topological zeta function
for Lê-Yomdin singularities. From now on, we consider $p=0$ and we switch from $G$ to $F$.
Theorem~\ref{thmfsuspACNLM} contains formulas which express the local topological zeta function of the suspension $F$  of the hypersurface $f$ by $\ex $ points in terms of the local topological zeta function and local twisted topological zeta functions of $f$ and $\ex $. The following  arithmetic lemma is needed.

\begin{lemma}\label{lema:arit}
Let $\ex , l
\in \NN$ such that $l$ does not divide $\ex $.
Let $l_1:=\frac{l}{\gcd (l, \ex )}$.
Then,
the set
\[
D(\ex , l_1
):=\left\{ M\in\NN\ \middle|\
l_1\gcd\left(\ex , M\right) \mid M\right\}
\]
satisfies that
\begin{enumerate}[label=\rm(\arabic{enumi}), series=arit]
\item\label{arit-1} if $M \in D(\ex , l_1)$ then $l_1 \mid M$,
\item\label{arit-2} if $M_1, M_2\in D(\ex ,l_1
)$ then $\gcd(M_1,M_2)\in D(\ex ,l_1)$,
\item\label{arit-3} If $M_1 \in D(\ex ,l_1)$ and $M_1 \mid M_2$ then $M_2 \in D(\ex , l_1)$.
\end{enumerate}
As a consequence there exists $m\in\NN$ such that $D(\ex , l_1)=\{k l_1 m\mid k \in\NN\}$.
Moreover,
\[
m=\min_{n\in\NN}\gcd(\ex ,l_1^n).
\]
\end{lemma}

The proof of Lemma \ref{lema:arit} will be given after Example \ref{ex:LvP}.

\begin{theorem} \label{thmfsuspACNLM}
Consider
$F(\mathbf{\boldsymbol{x}},z)=  z^\ex - f(\mathbf{x})$, and
$\omega_{d+1} = \mathbf{\boldsymbol{x}}^{\boldsymbol{\nu}} z \frac{\dif \mathbf{\boldsymbol{x}}}{\mathbf{x}} \frac{ \dif z}{z}$, i.e.,  with $p=0$ and $\nu_z=1$.
The following explicit expressions for the local and the local $l$-twisted
Denef-Loeser topological zeta functions  of the suspension hold (recall $t=s+\frac{1}{\ex }$).
\begin{enumerate}[label=\rm(\alph{enumi})]
\item For $l=1$ we have
\begin{equation}
\label{fsuspACNLMth}
\frac{s+1}{s}\Ztop{}(F,\omega_{d+1},s)_\zero=
\frac{s+1}{\ex st}+\frac{\ex -1}{\ex }\frac{t + 1}{t}
\Ztop{}(f,\omega_d,t)_\zero
	- \sum_{1\neq e | \ex } \frac{J_2(e)}{\ex } \Ztop{(e)}(f,\omega_d,t)_\zero
\end{equation}

\item For $1\neq l\mid \ex $ we have
\begin{equation}
	\label{twisted1fsuspACNLMth}
\Ztop{(l)}(F,\omega_{d+1},s)_\zero=  \! \frac{1}{\ex t}+
\Ztop{(l)}(f,\omega_d,t)_\zero \!- \! \frac{t+1}{\ex t}\Ztop{}(f,\omega_d,t)_\zero
	\!- \! \sum_{ 1\neq e \mid \ex } \frac{J_2(e)}{\ex } \Ztop{(e)}(f,\omega_d,t)_\zero,
\end{equation}

\item For $l\nmid \ex $ we have
\begin{equation}
	\Ztop{(l)}(F,\omega_{d+1},s)_\zero=
\Ztop{(l)}(f,\omega_d,t)_\zero
	- \! \sum_{e | \ex } \frac{J_2(e)}{\ex }
	\Ztop{(\lcm(e,m))}(f,\omega_d,t)_\zero.
\label{twisted2fsuspACNLMth}
\end{equation}

\end{enumerate}

\end{theorem}

\begin{remark}Note that \eqref{fsuspACNLM} is equivalent to \eqref{fsuspACNLMth}.\end{remark}

\begin{proof}The statement \eqref{fsuspACNLMth} is a consequence of Theorem \ref{thmpsuspACNLM}. Taking $p=0$ the polynomial $G$ becomes $F$. Moreover taking $p=0$ and $\nu_z=1$ the parameter $r$ becomes $t$ and $U(r,s)$ becomes $U(t,s) =   1 -  \frac{1}{\ex t} - \frac{s}{s+1} \frac{1}{\ex} = \frac{s}{s+1} \frac{\ex-1}{\ex}\frac{t+1}{t}$. Finally, multiplying by $\frac{s+1}{s}$ we get \eqref{fsuspACNLMth}. See Remark \ref{matrix} for the motivation of the factor $\frac{s+1}{s}$.

To compute the local $l$-twisted topological zeta function  $\Ztop{(l)}(F,\omega_{d+1},s)_\zero$ we use the identity \eqref{eq:tzetatopW} and Corollary~\ref{twistedaportop}. If $l \mid \ex$,  we have
\[
\Ztop{(l)}(F,\omega_{d+1},s)_\zero=
\sum_{\emptyset \ne I \subset J}
\left(\frac{\chi(E_I^\circ)}{\displaystyle \ex t\prod_{i\in I} \nu_i}
+\left ( -\frac{1}{\displaystyle \ex t}-\frac{e_I^2}{\ex }\right)
\frac{\chi(E_I^\circ)}{
\displaystyle
\prod_{i\in I} (N_i t + \nu_i)}
\right) +
\sum_{\substack{ \emptyset \ne I \subset J\\  l \mid  N_I}} \frac{\chi(E_I^\circ)}{
\displaystyle
\prod_{i\in I} (N_i t + \nu_i)}
\]
Then, using \eqref{DLz(0)=1} and the expression of the local zeta function in terms of an embedded resolution \eqref{eq:zetas_res}, we have
\begin{equation*}
\Ztop{(l)}(F,\omega_{d+1},s)_\zero=
\frac{1}{\ex t}+
\Ztop{(l)}(f,\omega_{d},t)_\zero
-\frac{\Ztop{}(f,\omega_{d},t)_\zero}{\ex t}
-\frac{1}{\ex }\sum_{\emptyset \ne I \subset J}
\frac{e_I^2\chi(E_I^\circ)}{\displaystyle\prod_{i\in I} (N_it + \nu_i)},
\end{equation*}
which gives \eqref{twisted1fsuspACNLMth} with the same arguments
as in the proof of Theorem \ref{thmpsuspACNLM}.

The identity \eqref{eq:tzetatopW} and Corollary~~\ref{twistedaportop} are used again for the formula of
 $\Ztop{(l)}(F,\omega_{d+1},s)_\zero$ with $l \nmid \ex $. In this case, we have
\begin{equation*}
\Ztop{(l)}(F,\omega_{d+1},s)_\zero=
\Ztop{(l)}(f,\omega_{d},t)_\zero-
\sum_{\substack{\emptyset \ne I \subset J\\ l \mid \lcm(\ex ,N_I) }}
\frac{e_I^2\chi(E_I^\circ)  }{
\displaystyle \ex  \prod_{i\in I} (N_i t + \nu_i)}.
\end{equation*}
Let us rewrite the last summand as
\begin{equation}\label{eq:auxproof}
  \sum_{\substack{\emptyset \ne I \subset J\\ l \mid \lcm(\ex ,N_I) }}
\frac{e_I^2\chi(E_I^\circ)  }{
\displaystyle  \ex  \prod_{i\in I} (N_i t + \nu_i)} = \sum_{k | \ex } \frac{ k^2 }{\ex } \sum_{\substack {  \emptyset \ne I \subset J\\  k=e_I  \\ l \mid \lcm (\ex , N_I)} } \frac{\chi(E_I^\circ)  }{
\displaystyle  \prod_{i\in I} (N_i t + \nu_i)} .
\end{equation}
Since $l \mid \lcm (\ex , N_I)=\frac{N_I}{e_I}\ex $ is equivalent to  $l_1 \mid \lcm (\ex , N_I)=\frac{N_I}{e_I}\frac{\ex }{\gcd(\ex ,l)}$
and $l_1$ is coprime to $\frac{\ex }{\gcd(\ex ,l)}$,
applying \eqref{jk} to \eqref{eq:auxproof} we get
\begin{equation*}
\sum_{k \mid  \ex }  \sum_{e \mid k | \ex } \frac{ J_2(e) }{\ex } \sum_{\substack {  \emptyset \ne I \subset J\\  k= e_I \\ l_1k \mid N_I } } \frac{\chi(E_I^\circ)  }{
\displaystyle  \prod_{i\in I} (N_i t + \nu_i)} =    \sum_{e \mid \ex } \frac{ J_2(e) }{\ex }  \sum_{ e \mid k \mid  \ex } \sum_{\substack {  \emptyset \ne I \subset J\\  k= e_I \\ l_1k \mid N_I } } \frac{\chi(E_I^\circ)  }{
\displaystyle  \prod_{i\in I} (N_i t + \nu_i)}.
\end{equation*}
From Lemma~\ref{lema:arit} this term can be rewritten as
\begin{equation*}
\sum_{e \mid \ex } \frac{ J_2(e) }{\ex } \sum_{\substack {\emptyset \ne I \subset J\\
e\mid N_I\\
N_I\in  D(\ex , l_1
)
} } \frac{\chi(E_I^\circ)  }{
\displaystyle  \prod_{i\in I} (N_i t + \nu_i)}=
\sum_{e \mid \ex } \frac{ J_2(e) }{\ex } \sum_{\substack {\emptyset \ne I \subset J\\
e\mid N_I\\
l_1 m \mid N_I
} } \frac{\chi(E_I^\circ)  }{
\displaystyle  \prod_{i\in I} (N_i t + \nu_i)}.
\end{equation*}
Indeed, the conditions $e_I = k$ and $l_1k \mid N_I$ are equivalent to $N_I \in D(\ex , l_1)$
because of the definitions of $e_I$ and $D(\ex , l_1)$. And the condition $N_I \in D(\ex , l_1)$
is equivalent to $l_1 m \mid N_I $ because of the last statement
of Lemma~\ref{lema:arit}. Finally, using the properties of $\lcm$ we get~\eqref{twisted2fsuspACNLMth}.
\end{proof}

As we did for $G$, we can
bound the set of poles of $\Ztop{}(F,\dif x_1\wedge\dots\wedge\dif x_d\wedge\dif z,s)$.

\begin{cor}
The set $\pol(F,\dif x_1\wedge\dots\wedge\dif x_d\wedge\dif z)$ is contained in
\[
\left\{
-1, \frac{-1}{\ex}
\right\}
\cup
\left\{
\rho_0-\frac{1}{\ex}\,
\middle| \,
\rho_0\in\pol(f,\omega_d)
\right\}.
\]
\end{cor}

\begin{remark}\label{matrix} Equations \eqref{fsuspACNLMth} and \eqref{twisted1fsuspACNLMth} from Theorem \ref{thmfsuspACNLM} can be presented in a compact way using matrices. To see this consider a total order in the finite set $\mathcal{D}(\ex )$ of divisors of $\ex$ such that 1 is the smallest element, e.g.,  the usual order $\leq$ in $\NN$. Denote by $l_i$ the $i$-th element of $\mathcal{D}(\ex )$ under this order. We take the following $\abs{\mathcal{D}(\ex)}$-vectors:
\begin{itemize}
\item $ZF(s)$  whose first entry is   $\frac{s+1}{s}\Ztop{} (F,\omega_{d+1},s)$  and the other entries are  $\Ztop{(l_i)}(F,\omega_{d+1},s)$.
\item $Zf(t)$  whose first entry is   $\frac{t+1}{t}\Ztop{} (f,\omega_{d},t)$  and the other entries are  $\Ztop{(l_i)}(f,\omega_{d},t)$.
\item $A  = (\frac{s+1}{s}, 1, 1, \dots, 1)^t$.
\item $J_2$ whose $i$th entry is $J_2(l_i)$.
\end{itemize}
Finally, set $B$ for the square matrix of size $\abs{\mathcal{D}(\ex)}$ defined by $\ex  \cdot\mathbf{1}_{\abs{\mathcal{D}(\ex)}} - J $, where $\mathbf{1}_{\abs{\mathcal{D}(\ex)}}$ is the identity matrix, and the rows of $J$ are all equal to the vector $J_2$.
Then, the statement of Theorem \ref{thmfsuspACNLM} becomes:
\[
\frac{1}{\ex }ZF(s)  =  \frac{1}{t}A + B  \cdot Zf(t).
\]
\end{remark}

\begin{example}\label{ex:5-6-10}
Consider $f= x^5 + y^6$ and $\ex =10$. We have:
\[
\Ztop{(d)}(f,s) =
\begin{cases}
\frac{ 10 s + 11}{(30 s + 11)(s +1)}, & d=1\\
 \frac{4}{30 s + 11},  &d=2, 3, 6\\
\frac{5}{30s + 11}, &d=5\\
\frac{ -1}{30s + 11}, &d=10, 15, 30\\
\end{cases}
\quad \text{ and } \quad
\Ztop{(d)}(F,s) =
\begin{cases}
\frac{3s + 7}{(15 s +7) (s +1 )},& d=1\\
\frac{6}{15s + 7} & d=3,6\\
\frac{1}{2(15s + 7)}, &d=5\\
\frac{-5}{2(15s + 7)}, &d=10\\
\frac{7}{2(15s + 7)}, & d=15, 30\\
\end{cases}
\]
For other values of $d$, the zeta functions vanish.
Therefore, for $t = s+ \frac{1}{10}$, we have that
\[
10
\begin{pmatrix}
\frac{s+1}{s} \Ztop{}(F,s)_0 \\
\Ztop{(2)}(F,s)_0 \\
\Ztop{(5)}(F,s)_0 \\
\Ztop{(10)}(F,s)_0
\end{pmatrix}
=
\frac{1}{t}
\begin{pmatrix}
\frac{s+1}{s}\\
1\\
1\\
1
\end{pmatrix}
+
\begin{pmatrix}
9&-3&-24&-72\\
-1&7&-24&-72\\
-1&-3&-14&-72\\
-1&-3&-24&-62\\
\end{pmatrix}
\begin{pmatrix}
\frac{t+1}{t} \Ztop{}(f,t)_0 \\
\Ztop{(2)}(f, t)_0 \\
\Ztop{(5)}(f, t)_0 \\
\Ztop{(10)}(f, t)_0
\end{pmatrix}
\]
Let us see what happens for the values $l = 3,6,15,30 \nmid \ex =10$.  Notice that $l_1=3$ always, and $D(\ex , l_1) = \{3,6,15,30\}$ with $3$ as its minimum. Hence,
\[
\begin{aligned}
\Ztop{(l)}(F,s)_\zero &= \Ztop{(l)}(f,t)_\zero - \frac{1}{10} \Ztop{(3)}(f,t)_\zero -   \frac{3}{10} \Ztop{(6)}(f,t)_\zero  - \frac{24}{10} \Ztop{(15)}(f,t)_\zero - \frac{72}{10} \Ztop{(30)}(f,t)_\zero  \\~~~~+
&= \Ztop{(l)}(f,t)_\zero - \frac{4}{10} \Ztop{(3)}(f,t)_\zero - \frac{96}{10} \Ztop{(15)}(f,t)_\zero .
\end{aligned}
\]
\end{example}

\begin{example} \label{ex:LvP}
Consider the polynomial $f=x^{12} + y^{13} + z^{14} + xy^7 + y^2z^4 + x^2y^2z^3$ from \cite[Example 2]{MR2806692} and $\ex =84$.
The Newton polygon of $f$ contains a compact face determined by the monomials $xy^7$, $y^4z^4$, and $x^2y^2z^4$ and its contained in the affine hyperplane with equation $5p+7q+10r= 54$. For $\ex =84$ and $l=27 \nmid \ex =84$, we have that $l_1= 9$, and $m=3$. For $e \mid 84$ we have
\[
\lcm (e, 27) =
\begin{cases}
27e& \text { if } 3 \nmid e\\
9e& \text{ otherwise},
\end{cases}
\]
and
\[
\Ztop{(\lcm (e, 27) )}(f,\omega, s)_\zero =
\begin{cases}
\frac{1}{2 (27s + 11)}& \text { if }  e = 1,2,3,6\\
0& \text{ otherwise},
\end{cases}
\]
Therefore, we have
\[
\begin{aligned}
\Ztop{(27)}(F,s) &=  \Ztop{(27)}(f,t) - \sum_{e \mid 84} \frac{J_2(e)}{84} \Ztop{(\lcm(27,e))}(f,t) = \Ztop{(27)}(f,t) - \sum_{e \mid  6} \frac{J_2(e)}{84} \Ztop{(\lcm(27,e))}(f,t)
\\
&=  \Ztop{(27)}(f,t) - \frac{1}{84} \Ztop{(27)}(f,t)  -
\frac{3}{84} \Ztop{(54)}(f,t) - \frac{8}{84} \Ztop{(27)}(f,t) -
\frac{24}{84} \Ztop{(54)}(f,t)\\
 &= \frac{12}{21} \Ztop{(27)}f,t)  = \frac{8}{756s + 317}
 \end{aligned}
 \]
\end{example}

\begin{proof}[Proof of Lemma{\rm~\ref{lema:arit}}] The statement \ref{arit-1} follows directly from the definition of the set $D(\ex , l_1)$.

For any prime number $p$ let us denote the corresponding $p$-valuation by $\nu_p$.  Recall that if  $M_1 \mid M_2$ then $\nu_p(M_1) \leq  \nu_p(M_2)$. In particular, if  $M:=\gcd(M_1,M_2)$, then $\nu_p(M)=\min(\nu_p(M_1),\nu_p(M_2))$.

The property $M\in D(\ex , l_1)$ is equivalent to the property
\begin{equation}\label{valD}
\nu_p(l_1) + \min(\nu_p(\ex ),\nu_p(M))\leq\nu_p(M) \quad \text{ for all } p.
\end{equation}
Let us assume  $M_1,M_2\in D(\ex , l_1
)$. Then  $M=\gcd(M_1,M_2)$ satisfies that
\begin{gather*}
\nu_p(l_1) + \min(\nu_p(\ex ),\nu_p(M))=
\nu_p(l_1) + \min(\nu_p(\ex ),\min(\nu_p(M_1),\nu_p(M_2)))\\
=\nu_p(l_1) + \min(\nu_p(\ex ),\nu_p(M_1),\nu_p(M_2))\leq
\nu_p(l_1) + \min(\nu_p(\ex ),\nu_p(M_1))\leq\nu_p(M_1)
\end{gather*}
and analogously
$\nu_p(l_1) + \min(\nu_p(\ex ),\nu_p(M))\leq\nu_p(M_2)$.
Hence
\[\nu_p(l_1) + \min(\nu_p(\ex ),\nu_p(M))\leq \min (\nu_p(M_1), \nu_p(M_2)) = \nu_p(M).\]
Thus $M = \gcd(M_1, M_2) \in D(\ex ,l_1)$ and the statement \ref{arit-2} holds.

Let us assume that $M_1 \in D(\ex , l_1)$ and $M_1 \mid M_2$. If $\min ( \nu_p(\ex ), \nu_p(M_1)) = \nu_p(\ex ) $ then \eqref{valD} implies
$\nu_p(l_1) + \nu_p (\ex ) \leq \nu_p (M) $. Since $M_1 \mid M_2$ we have $\nu_p(l_1) + \nu_p (\ex ) \leq \nu_p (M_1) \leq \nu_p(M_2)$, i.e., $M_2 \in D(\ex , l_1)$. If instead one has, $\min ( \nu_p(\ex ), \nu_p(M_1)) = \nu_p(M_1)$, the latter and  \eqref{valD}  imply that $\nu_p(l_1)=0$ and hence $M_2 \in D(\ex , l_1)$. We conclude that  statement \ref{arit-3} holds.

A set of natural numbers closed by taking multiples and $\gcd$'s is the set of multiples of a given number~$m$.
The formula for~$m$ is straightforward.
\end{proof}

\section{Comparison with previous work}\label{comp}

The formula for the suspension by two points $F= z^2 + f(x_1, \ldots, x_{d})$ of a hypersurface singularity $f$
was proved in \cite[Theorem 1.1]{ACNLM-JLMS}. Their formula is a special case of \eqref{fsuspACNLM} or
Theorem~\ref{thmfsuspACNLM}.
A correspondent suggested to Artal \emph{et al.} the following generalization for $F = z^\ex  - f(x_1, \ldots, x_{d})$, i.e. the suspension by $\ex$ points,
\begin{align}\label{fsuspACNLM-Loeser}
	\frac{\ex -1}{\ex }
	\frac{s}{s+1}
	\frac{t+1}{t}\cdot \Ztop{}(f,t)_0
	- \frac{s}{s+1} \sum_{1\neq e | \ex } \frac{(e+1) \phi(e)}{\ex } \Ztop{(e)}(f,t)_0 + \frac{1}{\ex t},
\end{align}
where $\phi$ is the Euler totient function. This expression is supposedly based on the motivic Thom-Sebastiani theorem~\cite{Denef-LoeserDuke},
see the \emph{Note added in proof} in~\cite{ACNLM-JLMS}.

It is however easy to check that  \eqref{fsuspACNLM-Loeser} is inaccurate. For instance, consider the polynomial $F(x,y,z) = x^\ex  + y^\ex  + z^\ex $ that is non degenerated with respect to its Newton polyhedra. According to  \cite[Theorem 5.3 or Example (5.4)]{DL-JAMS},  we have that
\begin{equation}\label{NonDegFormula}
\Ztop{}(F,s)_\zero= \frac{(\ex -2)(\ex -1)s+s+3}{(s+1)(\ex s+3)}.
\end{equation}
On the other hand, $F(x,y,z)$ is the suspension by $\ex $ points of the plane curve singularity defined by $f(y,z)=y^\ex +z^\ex $. An embedded resolution of $\{f=0\}$ is obtained by a single point blow-up. One has a unique exceptional divisor $E_1$ with numerical data $(N_1, \nu_1)=(\ex ,2)$ that is intersected by the $\ex $ components of the strict transform of $\{f=0\}$. Hence, using well known formulas for the local topological zeta functions, one gets that
\[
Z^{(e)}_{\mathrm{top}}(f,s)_\zero=
\begin{cases}
 \frac{2+(2-\ex )s}{(\ex s+2)(s+1)} & \text{ for } e =1,\\
\frac{2-\ex }{\ex s+2}& \text{ for } 1 \ne e | \ex ,\\
0 & \text{ otherwise.}
\end{cases}
\]
Applying the above formula \eqref{fsuspACNLM-Loeser} for the suspension one gets the expression
\begin{equation}\label{CalcSusp}
\frac{-\ex ^2s+(4s+3)\ex -2s}{\ex (s+1)(\ex s+3)} - \frac{s}{s+1} \sum_{e | \ex , e \ne 1} \frac{(e+1) \phi(e)}{\ex } \frac{2-\ex }{\ex s+3}.
\end{equation}
The difference between  \eqref{NonDegFormula}  and \eqref{CalcSusp} is given by
\begin{equation}\label{resta}
\frac{(\ex -2)s}{\ex (s+1)(\ex s+3)} \big ((\ex +1)(\ex -1) -  \sum_{e | \ex , e \ne 1} (e+1) \phi(e) \big )
\end{equation}
As a consequence of Gauss' identity \eqref{eq:Gauss}, the expression (\ref{resta}) vanishes if and only if $\ex $ is prime.
The agreement between \eqref{NonDegFormula} and \eqref{fsuspACNLM} or Theorem \ref{thmfsuspACNLM} follows from the fact that there is a unique strata with $1 \ne e_I$ and in this case $e_I=\ex $.

 \providecommand\noopsort[1]{}
\providecommand{\bysame}{\leavevmode\hbox to3em{\hrulefill}\thinspace}
\providecommand{\MR}{\relax\ifhmode\unskip\space\fi MR }
\providecommand{\MRhref}[2]{%
  \href{http://www.ams.org/mathscinet-getitem?mr=#1}{#2}
}
\providecommand{\href}[2]{#2}

\end{document}